\documentclass[10pt,leqeq]{article}
\usepackage{amssymb,amsthm,amsmath,amscd}
\usepackage[all]{xy}
\usepackage{booktabs}
\usepackage{hyperref}
\usepackage{url}
\usepackage{graphicx}
\usepackage{epstopdf}
\newtheorem{thm}{Theorem}
\newtheorem{lem}[thm]{Lemma}
\newtheorem{cor}[thm]{Corollary}
\newtheorem{prop}[thm]{Proposition}
\newtheorem{rem}[thm]{Remark}

\newtheorem{exmp}[thm]{Example}
\date{}

\begin{document}
\setlength{\baselineskip}{16pt}
\title{Boolean Differential Operators}
\author{Jorge Catumba \ \ and \ \ Rafael D\'iaz}
\maketitle

\begin{abstract}
We consider four  combinatorial interpretations for the algebra of Boolean differential operators and
construct, for each interpretation, a matrix representation for the algebra of Boolean differential operators.\\


\end{abstract}

\section{Introduction}

A Boolean function with $n$-arguments, $n \in \mathbb{N}$, is a map
$f: \mathbb{Z}_2^n \longrightarrow  \mathbb{Z}_2,$ where $\mathbb{Z}_2=\{0, 1 \}$ is the field with two
elements. The $\mathbb{Z}_2$-algebra $\mathrm{BF}_n$ of Boolean functions on $n$-arguments,
with pointwise sum and multiplication, is isomorphic to the Boolean algebra $\mathrm{P}\mathrm{P}[n]$ of sets of subsets of $[n]=\{1,...,n \}$.
Indeed, we identify a vector in $\mathbb{Z}_2^n$ with an element
of $\mathrm{P}[n]$ via the characteristic function, and  we identify a map  $\mathrm{P}[n] \longrightarrow  \mathbb{Z}_2$
with a subset of $\mathrm{P}[n]$ again with the help of characteristic functions. The sum and product of Boolean functions correspond to the symmetric difference and the intersection of subsets of $\mathrm{P}[n],$ respectively. The canonical isomorphism $\mathrm{BF}_n  \simeq  \mathrm{P}\mathrm{P}[n]$ just described establishes
the link between classical propositional logic and set theory \cite{b}. \\

The partial derivative $\partial_i f: \mathbb{Z}_2^n \longrightarrow  \mathbb{Z}_2$, for $i \in [n]$, of a Boolean function $f: \mathbb{Z}_2^n \longrightarrow  \mathbb{Z}_2$, see \cite{r}, is given  by
\begin{equation}\label{e1} \partial_i f(a) \ = \ f(a + e_i) + f(a),\end{equation}  where $a \in \mathbb{Z}_2^n$ and $e_i$ is the vector with $1$ at
the $i$-th position and $0's$ at the
other positions.\\

We define the $\mathbb{Z}_2$-algebra $\mathrm{BDO}_n$ of Boolean differential operators on $\mathbb{Z}_2^n$ in analogy with the definition of differential operators on the affine space $k^n$, for a field $k$ of characteristic zero, i.e. $\mathrm{BDO}_n$ is the subalgebra of $\mathrm{End}_{\mathbb{Z}_2}(\mathrm{BF}_n)$ generated by the operators of multiplication by Boolean functions, and the partial derivative operators $\partial_i$ defined in (\ref{e1}). \\

It turns out that $\mathrm{BDO}_n = \mathrm{End}_{\mathbb{Z}_2}(\mathrm{BF}_n)$, see \cite{d}, i.e. any $\mathbb{Z}_2$-linear operator from $\mathrm{BF}_n$ to itself is actually given by a Boolean differential operator. Therefore a Boolean differential operator $A \in \mathrm{BDO}_n$ is just a map
$$A:\mathrm{BF}_n \longrightarrow \mathrm{BF}_n \ \ \ \ \ \ \mbox{such that} \ \ \ \ \ A(f+g)=A(f) + A(g), \ \ \mbox{for} \ \ \ f,g \in \mathrm{BF}_n.$$

We are interested in finding a suitable set theoretical interpretation for the algebras $\mathrm{BDO}_n$ that extends the
above mentioned interpretation of $\mathrm{BF}_n$ as the Boolean algebra $\mathrm{P}\mathrm{P}[n]$, and may shed a light towards a logical understanding
of the $\mathbb{Z}_2$-algebras $\mathrm{BDO}_n$. Indeed, we believe that the $\mathbb{Z}_2$-algebras $\mathrm{BDO}_n$ may play a semantic role, analogous to that played by truth functions in classical logic, within the context of a "quantum like" operational logic yet to be fully understood. A few steps in that direction are taken in \cite{d}.  \\

Our main goal in this work is to find suitable matrix representations for the $\mathbb{Z}_2$-algebras $\mathrm{BDO}_n$. By dimension counting $\mathrm{BDO}_n = \mathrm{End}_{\mathbb{Z}_2}(\mathrm{BF}_n)$ is non-canonically isomorphic to the $\mathbb{Z}_2$-algebra $\mathrm{M}_{2^n\times 2^n}(\mathbb{Z}_2)$ of square matrices of size $2^n$ with $0$-$1$ entries. Note that $\mathrm{M}_{2^n\times 2^n}(\mathbb{Z}_2)$ may be identified, via
characteristic functions, with $\mathrm{P}(\mathrm{P}[n]\times \mathrm{P}[n])$
the set of subsets of $\mathrm{P}[n]\times \mathrm{P}[n]$, or equivalently, with the set $\mathrm{DG}_{\mathrm{P}[n]}$  of simple directed graphs (possibly with loops) with vertex set $\mathrm{P}[n].$ A matrix $A \in \mathrm{M}_{2^n\times 2^n}(\mathbb{Z}_2) = \mathrm{P}(\mathrm{P}[n]\times \mathrm{P}[n])$
is regarded as a directed graph with vertex set $\mathrm{P}[n]$ by drawing an edge from $b \in \mathrm{P}[n]$ to  $a \in \mathrm{P}[n]$ if and only if $(a,b) \in A$. The sum and product of matrices in $\mathrm{M}_{2^n\times 2^n}(\mathbb{Z}_2)$ induce operations of sum and product  of digraphs in $\mathrm{DG}_{\mathrm{P}[n]}$.
The sum on $\mathrm{DG}_{\mathrm{P}[n]}$ is the symmetric difference. The product $AB$ of digraphs $A, B \in \mathrm{DG}_{\mathrm{P}[n]}$ is such that the pair $(a,b) \in AB$ if and only if there is and odd
number of sets $c \in \mathrm{P}[n]$ such that $(a,c) \in A$ and $(c,b) \in B.$ \\

To define an explicit isomorphism $\mathrm{BDO}_n  \simeq  \mathrm{M}_{2^n\times 2^n}(\mathbb{Z}_2)$ a choice of basis for $\mathrm{BF}_n$ must be made. In this work we only consider the basis $\{\ m^a \ | \ a \in \mathrm{P}[n]\ \}$  for $\mathrm{BF}_n$, where the Boolean function
$m^a:\mathbb{Z}_2^n \longrightarrow  \mathbb{Z}_2$ is given on $b \in \mathrm{P}[n]$  by:
\begin{equation}\label{e20}
m^a(b)=
\left\{\begin{array}{cc} 1 & \ \mathrm{if} \ a = b,\\
0 & \ \mathrm{otherwise.}  \end{array}\right.
\end{equation}
We let $[A] \in \mathrm{M}_{2^n\times 2^n}(\mathbb{Z}_2)$ be the matrix of the Boolean differential operator $A \in \mathrm{BDO}_n$ in the basis $\{\ m^a \ | \ a \in \mathrm{P}[n]\ \}$. \\

We are going to use the following simple algebraic construction.
Let $A$ be a $\mathbb{Z}_2$-algebra, $V$ a $\mathbb{Z}_2$-vector space, and $l:V \longrightarrow A$ be a $\mathbb{Z}_2$-linear bijective map.
We use $l$ to pullback the product on $A$ to a product on $V$ given for $\ v,w \in V$ by
$$vw \ = \ l^{-1}(l(v)l(w)).$$  With this product on $V$ the map $l$ becomes an algebra isomorphism. \\

As we shall see each of our choices of bases for $\mathrm{BDO}_n$ induces a $\mathbb{Z}_2$-linear bijective map from $\mathrm{DG}_{\mathrm{P}[n]}$
to $\mathrm{BDO}_n$. In  \cite{d} we use four such bijections to pullback the composition product on $\mathrm{BDO}_n$ to
$\mathrm{DG}_{\mathrm{P}[n]}$, thus we obtain four products on $\mathrm{DG}_{\mathrm{P}[n]}$ denoted, respectively, by $\star, \ \circ,\ \bullet, \ \ast.$
Having various presentations for the product on $\mathrm{BDO}_n $  is desirable, just as it is useful to generate truth functions by several types of logical connectives.\\

Our main goal in this work is to explicitly describe matrix representations for the products $\star, \ \circ,\ \bullet, \ \ast$
on $\mathrm{DG}_{\mathrm{P}[n]}$. It turns out that the product $\star$ is the easiest to handle, in Section 2 we discuss some of its basic properties and describe an explicit isomorphism with $\mathrm{M}_{2^n\times 2^n}(\mathbb{Z}_2)$.
In the remaining Sections, we present explicit isomorphisms between the products $\circ,\ \bullet, \ \ast$ and
the product on  $\mathrm{M}_{2^n\times 2^n}(\mathbb{Z}_2)$, the algebra of square matrices of size $2^n$ with entries in  $\mathbb{Z}_2$.\\

Let us comment on some conventions assumed in this work. In the figures we draw a subset of $\mathrm{P}[n]\times \mathrm{P}[n]$ as a subset of the real plane, using the bijective correspondence between $\mathrm{P}[n]=\mathbb{Z}_2^n$ and the natural numbers in the interval $[0,2^n-1]$ resulting of ordering
$\mathrm{P}[n]$ by cardinality and lexicographic order within a given cardinality.
For example $\mathrm{P}[2]$ and $[0,3]$  are in correspondence as follows $\emptyset \rightarrow 0, \ \{1\} \rightarrow 1, \ \{2\} \rightarrow 2, \ \{1,2\} \rightarrow 3.$
When drawing a product, the elements of the first factor are drawn as triangles; the elements of the second factor are drawn as circles;
and the elements in the product are drawn as stars. We identify matrices in $\mathrm{M}_{2^n\times 2^n}(\mathbb{Z}_2)$ with maps
$\mathrm{P}[n] \times \mathrm{P}[n] \longrightarrow \mathbb{Z}_2$ using again the cardinality-lexicographic order on $\mathbb{Z}_2^n = \mathrm{P}[n].$ We use
juxtaposition for the product of matrices, and $\mathrm{rank}(A)$ for the rank of matrix $A$.

\section{MS Basis and the $\star$-Product}

As mentioned in the introduction  we are going to consider four different bases for the $\mathbb{Z}_2$-algebra $\mathrm{BDO}_n$ of Boolean differential operators
on $\mathbb{Z}_2^n$.
In this section we consider the $\mathrm{MS}$-basis $$\{ \ m^cs^d \ | \ c,d \in \mathrm{P}[n]\ \},$$ where the Boolean functions $m^c$ were described in the introduction, and the shift operators $s^d: \mathrm{BF}_n \longrightarrow \mathrm{BF}_n$ are given by
$$s^d = \prod_{i \in d}s_i, \ \ \ \ \mbox{where} \ \ \ \  s_if(a)=f(a + e_i), \ \  \mbox{for} \ \ a\in \mathbb{Z}_2^n, \ i \in [n], \ f \in \mathrm{BF}_n  .$$
Note that $\partial_i = s_i +1$ and  $s_i = \partial_i +1$, where $1$ stands for the identity operator; thus one can move back and forward
from the shift operators $s^d=\prod_{i \in d}s_i $ to the partial derivatives operators $\partial^d = \prod_{i \in d}\partial_i.$
Indeed, it is easy to check that
\begin{equation}\label{e6}\partial^d = \sum_{c \subseteq d}s^c \ \ \ \ \ \ \  \mbox{and} \ \ \ \ \ \ \ s^d = \sum_{c \subseteq d}\partial^c.\end{equation}

Consider the identifications
$$\mathrm{DG}_{\mathrm{P}[n]} \  \ \simeq  \ \ \mathrm{Map}(\mathrm{P}[n]\times \mathrm{P}[n], \mathbb{Z}_2)
\ \ \simeq \ \ \mathrm{BDO}_n,$$
where the identification on the left is given by characteristic functions and we use it freely without
change of notation; the non-canonical identification on the right is obtained via the bijective $\mathbb{Z}_2$-linear map
$l_1: \mathrm{DG}_{\mathrm{P}[n]} \rightarrow \mathrm{BDO}_n$
sending a directed graph $A \in \mathrm{DG}_{\mathrm{P}[n]}$ to the Boolean differential operator given by
$$l_1(A) \ = \   \sum_{(c,d) \in A}m^cs^d      \ = \ \sum_{c,d \in \mathrm{P}[n]}A(c,d)m^cs^d.$$
The $\star$-product on $\mathrm{DG}_{\mathrm{P}[n]}$ is the pullback via the map $l_1$ of the composition product on $\mathrm{BDO}_n$.
The $\star$-product, see \cite{d}, is given for $A,B \in \mathrm{DG}_{\mathrm{P}[n]}$  by the equivalent identities:
$$A \star B \ = \ l_1^{-1}(l_1(A)l_1(B));\ \ \  \ \ \ \ \ \  \ \ \ \  \ \ \ \ \ \ \ \ \ \ \ \ \ \ \ \ \ \ \  \ \ \ \ \ \ \ \ \ \ \ \ \ \ \ \  \ \ \ \ \ \ \ \ \ \ \ \ \ \ \ \ \ \ \ \ \ \ \ \ \ \ \ \ \ \
$$
\begin{equation}\label{e5} A \star B (c,d) \ = \ \sum_{e\in \mathrm{P}[n]}A(c,e)B(c+e,d+e);  \ \ \ \ \ \ \  \ \ \ \ \ \ \ \ \ \ \ \ \ \ \ \  \ \ \ \ \ \ \ \ \ \ \ \ \ \ \ \ \ \ \ \ \ \ \ \ \ \ \ \ \ \ \ \ \ \ \ \ \ \ \ \ \ \ \
\end{equation}
$$A \star B \ = \ \{ \ (c,d) \in \mathrm{P}[n]\times \mathrm{P}[n]\ \ | \ \ O\{e\in \mathrm{P}[n] \ | \ (c,e) \in A, \ \ (c+e,d + e) \in B \} \ \}, \ \ \ \ \  $$
where the notation $OC$ means that the finite set $C$ has odd cardinality.\\

We proceed to introduce a matrix representation for the algebra $ (\mathrm{DG}_{\mathrm{P}[n]}, \star)$.
Consider the map $M_1:\mathrm{DG}_{\mathrm{P}[n]}  \longrightarrow  \mathrm{M}_{2^n\times 2^n}(\mathbb{Z}_2) $ sending a directed graph $A \in \mathrm{DG}_{\mathrm{P}[n]} $ to the matrix of the operator $$l_1(A) \ = \ \sum_{c,d \in \mathrm{P}[n]}A(c,d)m^cs^d $$ in the basis $\{m^a \ | \ a \in \mathrm{P}[n]\},$ i.e.  we have that $$M_1(A)= [l_1(A)].$$  Note that $M_1$ is a $\mathbb{Z}_2$-linear map since it is the composition of two $\mathbb{Z}_2$-linear maps.

\begin{thm}
{\em The map $M_1:(\mathrm{DG}_{\mathrm{P}[n]}, \star)  \longrightarrow  (\mathrm{M}_{2^n\times 2^n}(\mathbb{Z}_2), \ . )$ is an
algebra isomorphism given for $A \in \mathrm{DG}_{\mathrm{P}[n]}$ by
\begin{equation}\label{e3}
M_1(A)_{a,b}\ = \ A(a, a+b).
\end{equation}
The inverse map $D_1: (\mathrm{M}_{2^n\times 2^n}(\mathbb{Z}_2), \ . ) \longrightarrow   (\mathrm{DG}_{\mathrm{P}[n]}, \star) $
sends a matrix $N \in \mathrm{M}_{2^n\times 2^n}(\mathbb{Z}_2)$ to the directed graph $D_1(N)$
 with characteristic function given by
\begin{equation}\label{e4}
D_1(N)(a,b) \ = \ N_{a,a+b}.
\end{equation}
}
\end{thm}

\begin{proof}
First note that  $\ m^cs^d(m^b)= \delta(b, c+d)m^c,\ $  where $\delta$ is the Kronecker delta function.
Therefore the matrix $[m^cs^d] \in \mathrm{M}_{2^n\times 2^n}(\mathbb{Z}_2)$ is such that:
\begin{equation}\label{e8}[m^cs^d]_{a,b}\ = \ \delta(a,c)\delta(b, c+d).\end{equation} Thus for $A \in \mathrm{DG}_{\mathrm{P}[n]}$ we have that:
$$M_1(A)_{a,b} \ = \  \left[\sum_{c,d \in \mathrm{P}[n]}A(c,d)m^cs^d \ \right]_{a,b} \ = \ \sum_{c,d \in \mathrm{P}[n]}A(c,d)[m^cs^d]_{a,b}\ =$$
$$\sum_{c,d \in \mathrm{P}[n]}A(c,d)\delta(a,c)\delta(b,c+d)\ = \ A(a, a+b).$$
The maps $M_1$ and $D_1$ are inverse of each other, indeed we have that
$$D_1(M_1(A))(a,b) \ = \ M_1(A)_{a,a+b}\ = \ A(a, a+a+b)\ = \ A(a,b),$$
$$M_1(D_1(N))_{a,b}\ = \ D_1(N)(a, a+b)\ = \ N(a, a+a+b)\ = \ N(a,b) $$
 for $A \in \mathrm{DG}_{\mathrm{P}[n]}$ and
$N \in \mathrm{M}_{2^n\times 2^n}(\mathbb{Z}_2)$. \\

Next we show that $M_1$ is an algebra morphism.  Thus we have that:
$$M_1(A \star B)\ = \ [l_1(A \star B)] \ = \  [l_1(A)l_1(B))] \ = \ [l_1(A)][l_1(B))] \ = \ M_1(A) M_1(B).$$
Explicitly, for $A, B \in \mathrm{DG}_{\mathrm{P}[n]} $,  we have using (\ref{e5}) and (\ref{e3}) that:
$$M_1(A \star B )_{a_1,a_2} \ = \ A \star B(a_1, a_1+a_2) \ =$$ $$ \sum_{b\in \mathrm{P}[n]}A(a_1,b)B(a_1+b,a_1 + a_2+b)\ = \
 \sum_{b\in \mathrm{P}[n]}M_1(A)(a_1,a_1 + b)M_1(B)(a_1+b, a_2)=$$
$$\sum_{b\in \mathrm{P}[n]}M_1(A)(a_1,b)M_1(B)(b, a_2)\ = \ \left(M_1(A)M_1(B)\right)_{a_1, a_2} .$$

\end{proof}

\begin{cor}
{\em  $|\mathrm{Ker}(l_1(A))|=2^r$ and $|\mathrm{Im}(l_1(A))|=2^{n-r}$, where $r=\mathrm{rank}(M_1(A))$.
}
\end{cor}

\begin{exmp}
{\em Consider the Jordan-like matrices $\mathrm{M}_{2^n\times 2^n}(\mathbb{Z}_2)$ having ones on the  principal diagonal and
on the diagonal directly above the principal. The associated Boolean differential operators in the $\mathrm{MS}$-basis,
for $n\in [4]$, are given in Table \ref{jordanBDO}.

\begin{exmp}
{\em The multiplication table of $ (\mathrm{DG}_{\mathrm{P}[1]}, \star)$ is given in Table \ref{table_product}.
}
\end{exmp}

\begin{table}[!hbtp]
\centering
\resizebox{12cm}{!}{
\begin{tabular}{lp{12cm}}
\toprule
 n & Operator \\
\toprule
 1 & $m^{\emptyset}s^{\{1\}} + 1$ \\
\midrule
 2 & $m^{\emptyset}s^{\{1\}} + m^{\{1\}}s^{\{1, 2\}} + m^{\{2\}}s^{\{1\}} + 1$ \\
\midrule
 3 & $m^{\emptyset}s^{\{1\}} + m^{\{1\}}s^{\{1, 2\}} + m^{\{2\}}s^{\{2, 3\}} + m^{\{3\}}s^{\{1, 2, 3\}} + m^{\{1, 2\}}s^{\{2, 3\}} + m^{\{1, 3\}}s^{\{1, 2\}} + m^{\{2, 3\}}s^{\{1\}} + 1$ \\
\midrule
 4 & $m^{\emptyset}s^{\{1\}} + m^{\{1\}}s^{\{1, 2\}} + m^{\{2\}}s^{\{2, 3\}} + m^{\{3\}}s^{\{3, 4\}} + m^{\{4\}}s^{\{1, 2, 4\}} + m^{\{1, 2\}}s^{\{2, 3\}} + m^{\{1, 3\}}s^{\{3, 4\}} + m^{\{1, 4\}}s^{\{1, 2, 3, 4\}} + m^{\{2, 3\}}s^{\{3, 4\}} + m^{\{2, 4\}}s^{\{2, 3\}} + m^{\{3, 4\}}s^{\{1, 2, 4\}} + m^{\{1, 2, 3\}}s^{\{3, 4\}} + m^{\{1, 2, 4\}}s^{\{2, 3\}} + m^{\{1, 3, 4\}}s^{\{1, 2\}} + m^{\{2, 3, 4\}}s^{\{1\}} + 1$\\
\bottomrule
 \end{tabular}}
 \caption{Associated Boolean differential operators of Jordan-like matrices.}
 \label{jordanBDO}
\end{table}
}
\end{exmp}

It would be nice to have an intuitive understanding of the  $\star$-product, say in the spirit of Venn diagrams. In order to gain a better understanding of the meaning of the $\star$-product we consider several examples. The first example is the simple case of the product of graphs with a unique edge.

\begin{lem}
{\em Let $a,b,c,d \in \mathrm{P}[n]$, then we have that:
$$\{(a,b)\}\star \{(c,d)\} \ = \
\left\{\begin{array}{cc}
\{(a,b+d)\} & \ \mathrm{if }\ \  a = b + c,\\
\emptyset & \  \mathrm{if} \ \ a \neq b+c. \end{array}\right.$$

}
\end{lem}

\begin{proof}
Let $(a_1, a_2)\in \{(a,b)\}\star \{(c,d)\} $, then there is an odd number of sets
$e \in \mathrm{P}[n]$ such that
$$(a_1,e) = (a,b) \ \ \ \ \ \mbox{and} \ \ \ \ \ (a_1 +e , a_2 +e)= (c,d),$$ Thus
$a_1=a, \ e =b,\  a=b+c, \ a_2=b+d.$ So, there is no such  $e$ if $a \neq  b+ c$. In the case $a = b+ c$,  we have that $(a_1,a_2) = (a,b+d) .$
\end{proof}

In the next examples we consider the $\star$-product on graphs using suitable decompositions of the graphs.

\begin{prop}
{\em Let $\ A, B \in \mathrm{DG}_{\mathrm{P}[n]} \ $ be given by
$$A = \sum_{b\in \mathrm{P}[n]}A_b\times \{b\} \ \ \ \ \  \mbox{and} \ \ \ \ \  B = \sum_{c\in \mathrm{P}[n]}B_c\times \{c\},$$
where $A_b$ is the set endpoints of edges in $A$ starting at $b$, and  $B_c$ is similarly defined.
Then we have that:
$$A \star B \ = \  \sum_{b,c  \in \mathrm{P}[n]}\left(A_b \cap (B_c + b)\right)\times \{b+c \} .$$
}
\end{prop}

\begin{proof} Using the distributive property, recall that
$(\mathrm{DG}_{\mathrm{P}[n]}, \star) \ \simeq \ (\mathrm{M}_{2^n\times 2^n}(\mathbb{Z}_2), \ . ),$ we get
$$A \star B \ = \  \sum_{b,c  \in \mathrm{P}[n]}\left(A_b\times \{b\}\right) \star \left(B_c\times \{c\}\right).$$
Suppose $(a_1, a_2) \in \left(A_b\times \{b\}\right) \star \left(B_c\times \{c\}\right)$, then
there is an odd number of sets $e \in \mathrm{P}[n]$ such that $(a_1, e) \in A_b\times \{b\}$ and $(a_1 + e, a_2 + e) \in  B_c\times \{c\}$.
Clearly, the only possible  set  $e$ with those properties is $e=b$,  and furthermore $a_1 \in A_b$, $a_2=b+c$, and $a_1 \in B_c + b.$ Therefore, we have that:
$$\left(A_b\times \{b\}\right) \star \left(B_c\times \{c\}\right)\ = \
\left(A_b \cap (B_c + b)\right)\times \{b+c \},$$ yielding the desired result.
\end{proof}

\begin{rem} {\em To contrast the classical intersection $\cap$ with the $\star$-product note that the intersection of  graphs $A,B \in \mathrm{DG}_{\mathrm{P}[n]}$ can be written as
$$A \cap B \ = \  \sum_{b \in \mathrm{P}[n]}\left(A_b \cap B_b \right)\times \{b\}
 \ = \  \sum_{b,c  \in \mathrm{P}[n]}\delta(b,c)\left(A_b \cap B_b\right)\times \{b\}.$$
}
\end{rem}

\begin{exmp}\label{e4}
{\em Let $A_{\{1,2\}}\ = \ \{\{1\}, \{2\}, \{3\}, \{4\}, \{1, 2\}, \{1, 3\}\} \in \mathrm{P}\mathrm{P}[4]$, then
$$(A_{\{1,2\}} \times \{1,2\})\star (A_{\{1,2\}} \times \{1,2\})\ = \ \{\{1\}, \{2\}\} \times \{\emptyset\} ,$$
since $\{1,2\} + \{1,2\} = \emptyset,\ $
$\ \{\{1\}, \{2\}, \{3\}, \{4\}, \{1, 2\}, \{1, 3\}\} + \{1,2\} $ is equal to $$\{\emptyset, \{1\}, \{2\}, \{2, 3\}, \{2,3\}, \{1,2,4\}\}=\{\emptyset, \{1\}, \{2\}, \{1,2,4\}\},$$
$$\mbox{and} \ \ \ \ \{\{1\}, \{2\}, \{3\}, \{4\}, \{1, 2\}, \{1, 3\}\}  \cap  \{\emptyset, \{1\}, \{2\}, \{1,2,4\}\} =   \{\{1\}, \{2\}\}  .$$
Figure \ref{fig:star_uno} shows a graphical representation of the product $(A_{\{1,2\}} \times \{1,2\})\star (A_{\{1,2\}} \times \{1,2\})$.
}
\begin{figure}[!h]
\centering
\includegraphics[scale=0.25]{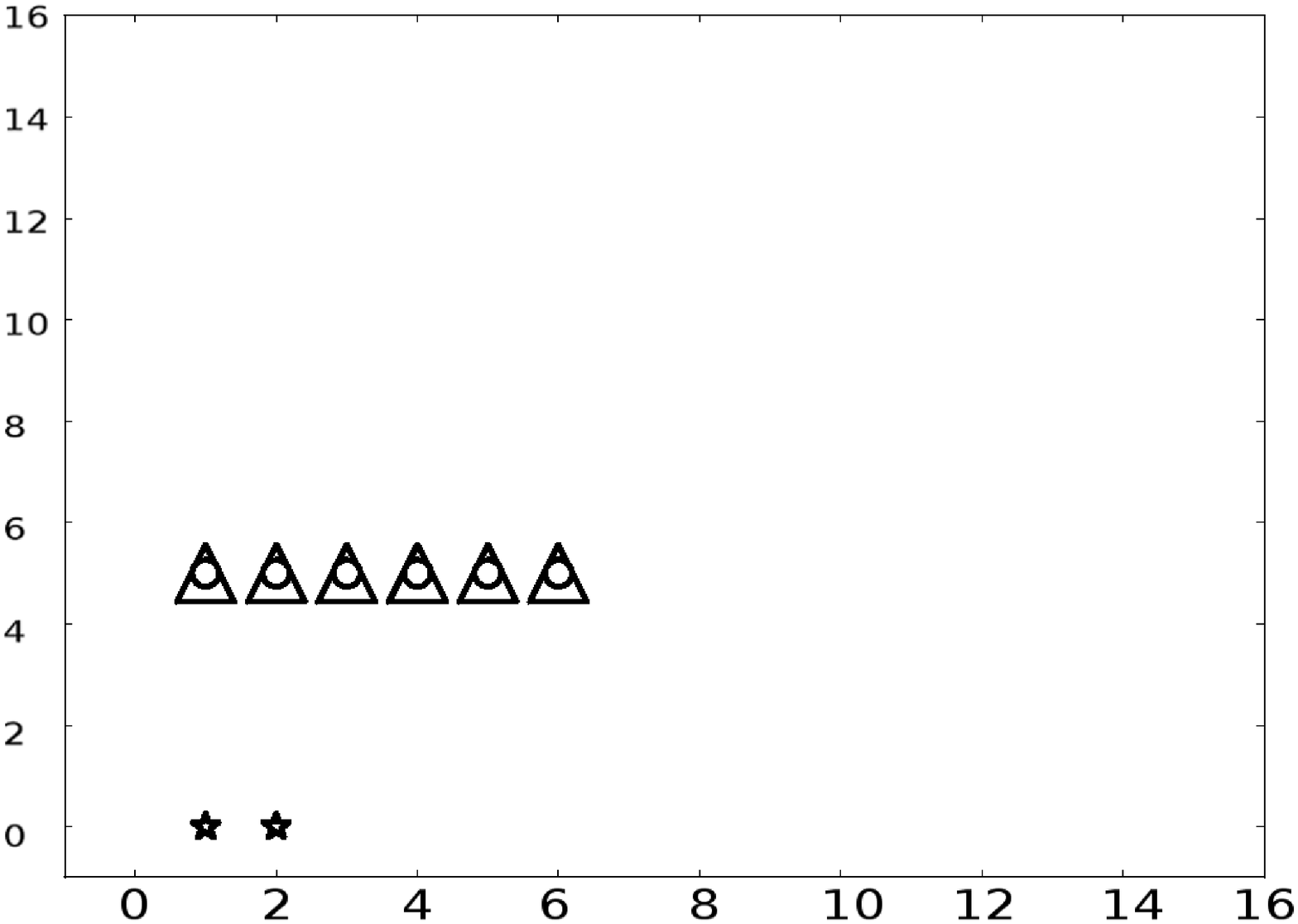}
\caption{Representation of the product $(A_{\{1,2\}} \times \{1,2\})\star (A_{\{1,2\}} \times \{1,2\})$,  Example \ref{e4}.}
\label{fig:star_uno}
\end{figure}
\end{exmp}

\begin{prop}
{\em Let $A, B \in \mathrm{DG}_{\mathrm{P}[n]}$ be written as
$$A = \sum_{b\in \mathrm{P}[n]}A_b\times \{b\} \ \ \ \ \ \ \mbox{and} \ \ \ \ \ \  B = \sum_{c\in \mathrm{P}[n]}\{c\} \times B_c.$$
Then we have that:
$$A \star B \ = \  \sum_{b+c \in A_b}\{b+c \} \times (B_c + b) .$$
}
\end{prop}

\begin{proof} Using the distributive property for the $\star$-product we get:
$$A \star B \ = \  \sum_{b,c \in \mathrm{P}[n]}\left(A_b\times \{b\}\right) \star \left(\{c\} \times B_c\right) \ = \ \sum_{b+c \in A_b}\{b+c \} \times (B_c + b).$$ We show the second identity.
Suppose $(a_1, a_2) \in  \left(A_b\times \{b\}\right) \star \left( \{c\} \times B_c \right),$
then there is a odd number of sets $\ e \in \mathrm{P}[n]$ such that $(a_1, e) \in A_b\times \{b\}$ and $(a_1 + e, a_2 + e) \in   \{c\} \times B_c$.
Clearly, the only possibility is $e=b$,  and furthermore $b +c = a_1 \in A_b$,  $\ a_2 \in B_c + b.$
\end{proof}

\begin{exmp}\label{hoy}
{\em Let $A_{\{1,2\}}$  and $B_{\emptyset}$ in $\mathrm{P}\mathrm{P}[4]$ be given, respectively, by
 $$A_{\{1,2\}} \ =\ \{\{1\}, \{2\}, \{3\}, \{4\}, \{1, 2\}, \{1, 3\}\}$$ and $$B_{\emptyset} \ =\ \{\{2\}, \{3\}, \{4\}, \{1,2\}, \{1,3\}, \{1,4\}\}.$$
Then we have, see Figure \ref{fig:star_dos}, that:
$$(A_{\{1,2\}}  \times \{\{1,2\}\})\star (\{\emptyset\}\times B_{\emptyset})\ = \ \{\{1,2\}\}\times\{\emptyset, \{1\}, \{2,3\}, \{2,4\}, \{1,2,3\}, \{1,2,4\}\}.$$

\begin{figure}[!h]
\centering
\includegraphics[scale=0.23]{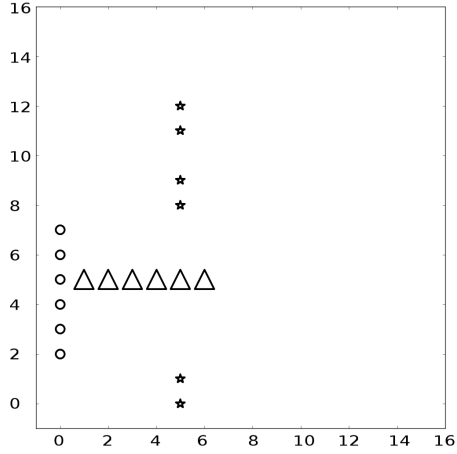}
\caption{Representation of the product $(A_{\{1,2\}}  \times \{1,2\})  \star (\emptyset\times B_{\emptyset})$ from Example \ref{hoy}.}
\label{fig:star_dos}
\end{figure}
}
\end{exmp}

\begin{prop}
{\em Let $A, B \in \mathrm{DG}_{\mathrm{P}[n]}$ be written as
$$A = \sum_{b\in \mathrm{P}[n]}\{b\}\times A_b \ \ \ \ \  \mbox{and} \ \ \ \ \   B = \sum_{c\in \mathrm{P}[n]}\{c\}\times B_c.$$
Then we have that:
$$A \star B \ = \  \sum_{b +c \in A_b}\{b \} \times \left(B_c + b+c\right) .$$
}
\end{prop}

\begin{proof}
Using the distributive property for the $\star$-product we have that:
$$A \star B \ = \ \sum_{b,c \in \mathrm{P}[n]}(\{b\}\times A_b) \star (\{c\}\times B_c) .$$
A pair $(a_1,a_2) \in \mathrm{P}[n] \times \mathrm{P}[n]$ belongs to $(\{b\}\times A_b) \star (\{c\}\times B_c)$ if there is an odd number of sets
$e \in \mathrm{P}[n]$ such that $(a_1, e) \in \{b\}\times A_b$ and $(a_1 +e , a_2 + e) \in \{c\}\times B_c$. Thus we have that
$a_1=b, \ e \in A_b, \ e=b+c, \ a_2 \in B_c +b +c.$ The desired result follows.

\end{proof}

\begin{exmp}\label{mana}
{\em  Let $B_{\{1,3\}}$ and $A_{\emptyset}$ in $\mathrm{P}\mathrm{P}[4]$ be given by
$$A_{\emptyset} = \{\{2\}, \{3\}, \{4\}, \{1,2\}, \{1,3\}, \{1,4\}\}$$ and $$B_{\{1,3\}} = \{\emptyset, \{1\}, \{2\}, \{3\}, \{4\}, \{1,2\}\}.$$
Then, see Figure \ref{fig:star_tres}, we have that:
$$(\{\emptyset\}\times A_{\emptyset})\star (\{\{1,3\}\}\times B_{\{1,3\}})\ = \ \{\emptyset\} \times \{\{1\}, \{3\}, \{1,3\}, \{2,3\}, \{1,2,3\}, \{1,3,4\}\}.$$
\begin{figure}[!h]
\centering
\includegraphics[scale=0.28]{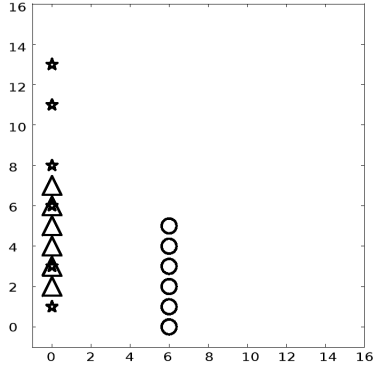}
\caption{Representation of the product $(\{\emptyset\}\times A_{\emptyset})\star (\{\{1,3\}\}\times B_{\{1,3\}})$, Example \ref{mana}.}
\label{fig:star_tres}
\end{figure}
}
\end{exmp}

\begin{exmp}\label{e2}
{\em Let $A, B \in \mathrm{DG}_{\mathrm{P}[4]}$ be given by
$$A\ =\ \{(\{4\},\{4\}), (\{4\},\{1,2\}), (\{4\},\{1,3\}), (\{1,2\},\{4\}), (\{1,2\},\{1,2\}),$$
$$(\{1,2\}, \{1,3\}), (\{1,3\},\{4\}), (\{1,3\},\{1,2\}), (\{1,3\},\{1,3\})\}, \ \ \ \ \ \mbox{and}$$
$$B\ =\ \{(\{1,3\},\{1,3\}), (\{1,3\},\{1,4\}), (\{1,3\},\{2,3\}), (\{1,4\},\{1,3\}), (\{1,4\},\{1,4\}),$$
 $$(\{1,4\},\{2,3\}), (\{2,3\},\{1,3\}), (\{2, 3\}, \{1, 4\}), (\{2, 3\}, \{2, 3\})\}.$$  The product $A\star B$, see Figure \ref{fig:star_cuatro},
is given by
{\small
 $$\{(\{1,2\},\emptyset), (\{1,2\},\{1,2\}), (\{1,2\},\{3,4\}), (\{1,3\},\{1,3\}), (\{1,3\},\{2,3\}), (\{1,3\},\{2,4\})\}.$$}
 \begin{figure}[!h]
\centering
\includegraphics[scale=0.28]{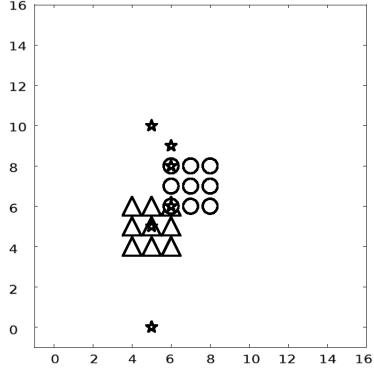}
\caption{Graphical representation of the product of the digraphs from Example \ref{e2} .}
\label{fig:star_cuatro}
\end{figure}
}
\end{exmp}

\section{M$\partial$ Basis and $\circ$-Product}

In this section we consider the $\mathrm{M\partial}$ basis for $\mathrm{BDO}_n$, thus
we regard a directed graph $A \in \mathrm{DG}_{\mathrm{P}[n]}$ as a Boolean differential operator via the map
$l_2: \mathrm{DG}_{\mathrm{P}[n]} \longrightarrow \mathrm{BDO}_n$ given by
$$l_2(A) \ = \ \sum_{(c,d) \in A}m^c\partial^d \ = \ \sum_{c,d \in \mathrm{P}[n]}A(c,d)m^c\partial^d,$$
where  for $d \in \mathrm{P}[n]$ we set $\partial^d= \prod_{i \in  d}\partial_i.$ \\

The $\circ$-product on $\mathrm{DG}_{\mathrm{P}[n]}$ is the pullback via the map $l_2$ of the composition product on $\mathrm{BDO}_n$, i.e.
the $\circ$-product is given for $A,B \in \mathrm{DG}_{\mathrm{P}[n]}$ by $A \circ B = l_2^{-1}(l_2(A)l_2(B))$. Explicitly, see \cite{d}, we have that
$$A \circ B (c,d)\ = \ \underset{{\underset {d\setminus g \subseteq c+f \subseteq  e}{ g \subseteq d, e, f}}} {\sum}A(c,e)B(f,g).$$
Equivalently, a pair $(c,d) \in \mathrm{P}[n]\times \mathrm{P}[n]$ belongs to $A \circ B \in \mathrm{DG}_{\mathrm{P}[n]}$ if and only if there is an odd number of sets  $e\in \mathrm{P}[n]\ $ and $\ (f, g) \in B$ such that
$$g \subseteq d , \ \ \ \ \ \ (c,e)  \in A, \ \ \ \ \ \  d\setminus g \subseteq c + f \subseteq  e.$$

Note that we can go back and forward from the $\mathrm{MS}$-basis to the $\mathrm{M}\partial$-basis for Boolean differential operators as follows:
$$A \ = \ \sum_{c,d \in \mathrm{P}[n]}A(c,d)m^c\partial^d\ =\ \sum_{c,d \in \mathrm{P}[n]}\widehat{A}(c,d)m^c s^d,$$
indeed with the help of the identities (\ref{e6}) we get that $$\widehat{A}(c,d)\ = \ \sum_{d \subseteq e}A(c,e) \ \ \ \ \ \ \  \mbox{and} \ \ \ \ \ \ \  A(c,d)\ = \ \sum_{d \subseteq e}\widehat{A}(c,e).$$

Consider the map $M_2:\mathrm{DG}_{\mathrm{P}[n]}  \longrightarrow  \mathrm{M}_{2^n\times 2^n}(\mathbb{Z}_2) $ sending
a directed graph $A \in \mathrm{DG}_{\mathrm{P}[n]} $ to the matrix of the operator
$$l_2(A)\ = \ \sum_{c,d \in \mathrm{P}[n]}A(c,d)m^c\partial^d \ = \ \sum_{c, e \subseteq d} A(c,d)m^cs^e $$
in the basis $\{m^a \ | \ a \in \mathrm{P}[n]\},$ i.e. we have that $M_2(A)=[l_2(A)]$.

\begin{thm}
{\em The map $M_2:(\mathrm{DG}_{\mathrm{P}[n]}, \circ )  \longrightarrow  (\mathrm{M}_{2^n\times 2^n}(\mathbb{Z}_2), \ .)$ is an
algebra isomorphism given for $A \in \mathrm{DG}_{\mathrm{P}[n]}$ by $$M_2(A)_{a,b}\ = \ \sum_{a+b \subseteq c}A(a, c). $$
The inverse map $D_2: (\mathrm{M}_{2^n\times 2^n}(\mathbb{Z}_2), \ .)  \longrightarrow   (\mathrm{DG}_{\mathrm{P}[n]}, \circ )$
sends  $N \in \mathrm{M}_{2^n\times 2^n}(\mathbb{Z}_2)$ to the directed graph $D_2(N)$ with characteristic function given by
$$D_2(N)(a,b) \ = \ \sum_{b \subseteq c} N_{a,a+c}.$$
}
\end{thm}

\begin{proof}
For $A \in \mathrm{DG}_{\mathrm{P}[n]}$, identifying $A$ with $l_2(A)$ and using ($\ref{e8}$), we get that
$$M_2(A)_{a,b}\ = \ \left[\sum_{c,d \in \mathrm{P}[n]}A(c,d)m^c\partial^d\right]_{a,b}  \ = \
\left[ \sum_{c, e \subseteq d} A(c,d)m^cs^e \right]_{a,b}$$
$$=\ \sum_{c, e \subseteq d} A(c,d)[m^cs^e]_{a,b}\ = \ \sum_{c,  e \subseteq d } A(c,d)\delta(a,c)\delta(b,c+e)=\sum_{a+b \subseteq d }A(a,d), $$
since $c=a$, and $b=c+e$ implies that $e=a+b.$\\

We show that $D_2$ is indeed the inverse of $M_2$. We have for $A \in \mathrm{DG}_{\mathrm{P}[n]}$ that:
$$D_2(M_2(A))(a,b) \ = \ \sum_{b \subseteq c} M_2(A)_{a,a+c}$$
$$=\ \sum_{b \subseteq c}\left( \sum_{a+a+c \subseteq d}A(a, d) \right) \ = \
\sum_{b \subseteq c \subseteq d}A(a, d)= A(a,b),$$ where the last identity is shown as follows:
$$ \sum_{b \subseteq c \subseteq d}A(a,d) \ = \ \sum_{b\subseteq d}
\left(\sum_{b \subseteq c \subseteq d}1\right)A(a,d) \ = \  \sum_{b\subseteq d}2^{|d \setminus b|}A(a,d) \ = \ A(a,b).$$

$M_2$ is an algebra morphism since $M_2(A)= [l_2(A)]$, and thus we have that:
$$M_2(A \circ B)\ = \ [l_2(A \circ B)] \ = \  [l_2(A)l_2(B))] \ = \ [l_2(A)] [l_2(B))] \ = \ M_2(A) M_2(B).$$
\end{proof}

\begin{cor}
{\em $|\mathrm{Ker}(l_2(A)|=2^r \ $ and $\ |\mathrm{Im}(l_2(A))|=2^{n-r}$, where $r=\mathrm{rank}(M_2(A))$.
}
\end{cor}

\section{XS-Basis and  the $\ast$-Product}
For $a \in \mathrm{P}[n]$ consider the Boolean function $x^a \in \mathrm{BF}_n$ given on $b \in \mathrm{P}[n]$  by:
\begin{equation}\label{e16}
x^a(b)\ = \
\left\{\begin{array}{cc} 1 & \mathrm{if} \ a \subseteq b,\\
0 & \ \  \mathrm{otherwise.}  \end{array}\right.
\end{equation}
One can go back and forward from the $\{m^a\}$ basis to the $\{x^a\}$ basis as follows:
\begin{equation}\label{e10} m^a = \sum_{a \subseteq b}x^b \ \ \ \ \ \ \  \mbox{and} \ \ \ \ \ \ \ x^a = \sum_{a \subseteq b}m^b.\end{equation}
Indeed, the identity on the right follows directly from definitions (\ref{e20}) and  (\ref{e16}); the identity on the left follows from the M$\ddot{\mathrm{o}}$bius inversion formula for modules over a $\mathbb{Z}_2$-ring, see \cite{ro}, which states that for arbitrary maps $$f,g: \mathrm{P}[n] \longrightarrow M,$$  with $M$ a module over a $\mathbb{Z}_2$-ring, the identities $$  g(d) = \sum_{c \subseteq d}f(c) \ \ \ \ \   \mbox{and} \ \ \ \ \  f(d) = \sum_{c \subseteq d}g(c) \ \ \ \mbox{are equivalent.}$$
Note that $x+x=2x=(1+1)x=0x=0$ for all $x\in M$.
The M$\ddot{\mbox{o}}$bius inversion formula follows from the identities (valid for $c \in \mathrm{P}[n]$ fixed):
\begin{equation}\label{e12}
\sum_{a \subseteq b \subseteq c}f(a)\ = \ \sum_{a\subseteq c}\left(\sum_{a \subseteq b \subseteq c}1\right)f(a) \ = \  \sum_{a\subseteq c}2^{|c \setminus a|}f(a) \ = \ f(c).
\end{equation}

In this section we regard elements of $\mathrm{DG}_{\mathrm{P}[n]}$ as Boolean differential operators via
the map $l_3: \mathrm{DG}_{\mathrm{P}[n]} \rightarrow \mathrm{BDO}_n$ given by
$$l_3(A) \ = \   \sum_{(c,d) \in A}x^cs^d      \ = \ \sum_{c,d \in \mathrm{P}[n]}A(c,d)x^cs^d.$$

The $\ast$-product on $\mathrm{DG}_{\mathrm{P}[n]}$ is the pullback via the map $l_3$ of the composition product on $\mathrm{BDO}_n$, i.e.
for $A,B \in \mathrm{DG}_{\mathrm{P}[n]}$ we have that $A \ast B = l_3^{-1}(l_3(A)l_3(B))$. Explicitly, see \cite{d}, we have that
$$A \ast B (c,d) \ = \ \sum_{e \subseteq c, g, h }
|\{k \subseteq  g\cap h \ | \ e \cup h\setminus k =c  \ \}|A(e,g)B(h,d+g).$$

Equivalently, a pair $(c,d) \in \mathrm{P}[n]\times \mathrm{P}[n]$ belongs to $A \ast B \in \mathrm{DG}_{\mathrm{P}[n]}$ if and only if there is an odd number of sets  $e,g,h,k \in \mathrm{P}[n]$  such that
$$e \subseteq c, \ \ \ \ \ \ \ k \subseteq g\cap h,\ \ \ \ \ \ \ e\cup h\setminus k= c, \ \ \ \ \ \ \  (e, g) \in A, \ \ \ \ \ \ \  (h,d + g) \in B .$$

Note that we can go back and forward from the $\mathrm{MS}$-basis to the $\mathrm{X}S$-basis for differential operators as follows:
$$A \ = \ \sum_{c,d \in \mathrm{P}[n]}A(c,d)x^cs^d\ =\ \sum_{c,d \in \mathrm{P}[n]}\widehat{A}(c,d)m^cs^d,$$
indeed using (\ref{e10}) we have that $$\widehat{A}(c,d)\ = \ \sum_{e \subseteq c}A(e,d) \ \ \ \ \ \ \ \ \mbox{and} \ \ \ \ \ \ \ \ A(c,d)\ = \ \sum_{e \subseteq c}\widehat{A}(e,d).$$

Consider the map $M_3:\mathrm{DG}_{\mathrm{P}[n]}  \longrightarrow  \mathrm{M}_{2^n\times 2^n}(\mathbb{Z}_2) $ sending
a directed graph $A \in \mathrm{DG}_{\mathrm{P}[n]} $ to the matrix of the operator
$$l_3(A)\ = \ \sum_{c,d \in \mathrm{P}[n]}A(c,d)x^cs^d \ = \ \sum_{c \subseteq e , d} A(c,d)m^es^d$$
in the basis $\{m^a \ | \ a \in \mathrm{P}[n]\},$ thus we have that $M_3(A)=l_3(A).$

\begin{thm}
{\em The map $M_3:(\mathrm{DG}_{\mathrm{P}[n]}, \ast)  \longrightarrow  (\mathrm{M}_{2^n\times 2^n}(\mathbb{Z}_2), \ .)$ is an
algebra isomorphism given for $A \in \mathrm{DG}_{\mathrm{P}[n]}$ by $$M_3(A)_{a,b}\ = \ \sum_{c \subseteq a}A(c,a+b). $$
The inverse map $D_3: (\mathrm{M}_{2^n\times 2^n}(\mathbb{Z}_2), \ .)  \longrightarrow   (\mathrm{DG}_{\mathrm{P}[n]}, \ast) $
sends $N \in \mathrm{M}_{2^n\times 2^n}(\mathbb{Z}_2)$ to the directed graph $D_3(N)$ with characteristic function given by
$$D_3(N)(a,b)\ = \  \sum_{c \subseteq a}N_{c,b+c} .$$
}
\end{thm}

\begin{proof}
We have that
$$M_3(A)_{a,b}\ = \ \left[\sum_{c,d \in \mathrm{P}[n]}A(c,d)x^cs^d \right]_{a,b}\ = \ \left[\sum_{c \subseteq e , d} A(c,d)m^es^d\right]_{a,b}$$
$$= \ \sum_{c \subseteq e , d} A(c,d)[m^es^d]_{a,b}\ =\ \sum_{c \subseteq e,  d } A(c,d)\delta(e,a)\delta(e,b+d)\ = \ \sum_{c\subseteq a }A(c,a+b). $$
The map $D_3$ is inverse to $M_3$ since for $A \in \mathrm{DG}_{\mathrm{P}[n]}$ we have that:
$$D_3(M_3(A))(a,b) \ = \ \sum_{c \subseteq a} M_3(A)_{c,b+c}$$
$$=\ \sum_{c \subseteq a}\left( \sum_{d \subseteq c}A(d, b) \right) \ = \
\sum_{d \subseteq c \subseteq a}A(d, b)= A(a,b).$$
We show that $M_3$ is an algebra morphism. By definition $M_3(A)= [l_3(A)],$ thus:
$$M_3(A \circ B)\ = \ [l_3(A \circ B)] \ = \  [l_3(A)l_3(B)] \ = \ [l_3(A)[l_3(B)] \ = \ M_3(A) M_3(B).$$
\end{proof}

\begin{cor}
{\em $|\mathrm{Ker}(l_3(A))|=2^r$ and $|\mathrm{Im}(l_3(A))|=2^{n-r}$, where $r=\mathrm{rank}(M_3(A))$.
}
\end{cor}

\section{X$\partial$-Basis and the $\bullet$-Product}

In this section we regard elements of $\mathrm{DG}_{\mathrm{P}[n]}$ as Boolean differential operators via the bijective map
$l_4: \mathrm{DG}_{\mathrm{P}[n]} \longrightarrow \mathrm{BDO}_n$ given by
$$l_4(A) \ = \   \sum_{(c,d) \in A}x^c\partial^d   \ = \ \sum_{c,d \in \mathrm{P}[n]}A(c,d)x^c\partial^d.$$

The $\bullet$-product on $\mathrm{DG}_{\mathrm{P}[n]}$ is the pullback via $l_4$ of the composition product on $\mathrm{BDO}_n$, thus
for $A,B \in \mathrm{DG}_{\mathrm{P}[n]}$ we have that $A \circ B = l_4^{-1}(l_4(A)l_4(B))$. Explicitly \cite{d} we have that:
{\footnotesize $$A\bullet B(c,d)\ = \
\underset{{e \subseteq c, h \subseteq d, f,  g}} {\sum}
|\left\{ k_1 \subseteq k_2 \subseteq f \cap g \ \left| \ e \cup ( g \setminus k_2) = c , f \setminus k_1= d \setminus h \right\}\right|
A(e,f)B(g,h).$$}

Equivalently, a pair $(c,d) \in \mathrm{P}[n]\times \mathrm{P}[n]$ belongs to $A \bullet B \in \mathrm{DG}_{\mathrm{P}[n]}$ if and only if there is an odd number of sets  $(e,f) \in A, \ (g,h) \in B,\ k_1 \subseteq k_2  \subseteq [n]$  such that
$$e \subseteq c,\ \ \ \ \  h \subseteq d, \ \ \ \ \  k_2 \subseteq f \cap g, \ \ \ \ \ e \cup ( g \setminus k_2) = c,
 \ \ \ \ \   f \setminus k_1= d \setminus h.$$

Note that we can go back and forward from the $\mathrm{MS}$-basis to the $\mathrm{X}\partial$-basis for Boolean differential operators, using equations (\ref{e6}) and (\ref{e10}) , as follows:
$$A \ = \ \sum_{c,d \in \mathrm{P}[n]}A(c,d)x^c\partial^d\ =\ \sum_{c,d \in \mathrm{P}[n]}\widehat{A}(c,d)m^c s^d,$$
where $$\widehat{A}(c,d)\ = \ \sum_{e \subseteq a, \ b \subseteq f}A(e,f) \ \ \ \ \ \ \ \mbox{and} \ \ \ \ \ \ \
A(c,d)\ = \ \sum_{e \subseteq a, \ b \subseteq f}\widehat{A}(e,f).$$

Consider the map $M_4:\mathrm{DG}_{\mathrm{P}[n]}  \longrightarrow  \mathrm{M}_{2^n\times 2^n}(\mathbb{Z}_2) $ sending
a directed graph $A \in \mathrm{DG}_{\mathrm{P}[n]} $ to the matrix of the operator
$$l_4(A)\ = \ \sum_{c,d \in \mathrm{P}[n]}A(c,d)x^c\partial^d \ = \ \sum_{c \subseteq e, \ f \subseteq d} A(c,d)m^es^f$$
in the basis $\{m^a \ | \ a \in \mathrm{P}[n]\}.$

\begin{thm}
{\em The map $M_4:\mathrm{DG}_{\mathrm{P}[n]}  \longrightarrow  \mathrm{M}_{2^n\times 2^n}(\mathbb{Z}_2)$ defines an
algebra isomorphism between $ (\mathrm{DG}_{\mathrm{P}[n]}, \bullet)$ and $\mathrm{M}_{2^n\times 2^n}(\mathbb{Z}_2)$. Explicitly,
for $A \in \mathrm{DG}_{\mathrm{P}[n]}$ we have that:
$$M_4(A)_{a,b}\ = \ \sum_{c \subseteq a, \ a+b \subseteq d} A(c,d). $$
The inverse map $D_4: \mathrm{M}_{2^n\times 2^n}(\mathbb{Z}_2)  \longrightarrow   \mathrm{DG}_{\mathrm{P}[n]}$
sends $N \in \mathrm{M}_{2^n\times 2^n}(\mathbb{Z}_2)$ to the  graph $D_4(N) \in \mathrm{DG}_{\mathrm{P}[n]}$ with characteristic function given by:
$$D_4(N)(a,b)\ = \ \sum_{c \subseteq a, \ b \subseteq d} N_{c, c+d}.$$
}
\end{thm}

\begin{proof}
Since $x^c = \sum_{c \subseteq e}m^e$ and $\partial^d = \sum_{f \subseteq d}s^f,$
we have that
$$M_4(A)_{a,b}\ = \  \left[\sum_{c,d}A(c,d)x^c\partial^d\right]_{a,b}  \ = \
\left[\sum_{c \subseteq e, \ f \subseteq d} A(c,d)m^es^f\right]_{a,b}  \ = $$
$$  \sum_{c \subseteq e, \ f \subseteq d} A(c,d)[m^es^f]_{a,b}\ = \
\sum_{c \subseteq e, \ f \subseteq d } A(c,d)\delta(e,a)\delta(e,f+b)=\sum_{c \subseteq a, \ a+b \subseteq d}A(c,d),$$
as $e=a$, and $e=f+b$ implies that $f=a+b.$ \\

The map $D_4$ is inverse of $M_4$ since applying $(\ref{e12})$ we have that:
$$D_4(M_4(A))(a,b) \ = \ \sum_{c \subseteq a, \ b \subseteq d} M_4(A)_{c, c+d}\ = $$
$$ \sum_{e \subseteq c \subseteq a, \ b \subseteq c+c+d \subseteq f} A(e,f) \
= \ \sum_{e \subseteq c \subseteq a, \ b \subseteq d \subseteq f} A(e,f) \ = \ A(a,b).$$
The map $M_4(A)= [l_4(A)]$ is an algebra morphism since:
$$M_4(A \bullet B)\ = \ [l_4(A \bullet B)] \ = \  [l_4(A) l_4(B))] \ = \ [l_4(A)] [l_4(B))] \ = \ M_4(A) M_4(B).$$
\end{proof}

\begin{cor}
{\em $|\mathrm{Ker}(l_4(A))|=2^r$ and $|\mathrm{Im}(l_4(A))|=2^{n-r},$ where $r=\mathrm{rank}(M_4(A))$.}
\end{cor}

\section{Final Comments}

In this work we considered four combinatorial interpretations using directed graphs for the composition (together with the symmetric difference) of Boolean differential operators, and
provided a matrix representation for each of these interpretations.  Therefore our work provides set theoretical interpretations for the algebra of Boolean differential operators.
It would be nice to find  logical interpretations as well, i.e. some sort of non-commutative logic where Boolean
differential operators play the role played by Boolean functions in classical propositional logic.
Partial results along this line are developed in \cite{d}, where a couple of explicit presentations by generators and relations of
the algebra of Boolean differential operators are provided.\\

We thank an anonymous referee for many valuable suggestions.

\bigskip

\bigskip

\noindent jorgerev90@gmail.com\\
\noindent Departamento de Matem\'aticas\\
\noindent Universidad Nacional de Colombia, Manizales, Colombia\\

\noindent ragadiaz@gmail.com\\
\noindent Instituto de Matem\'aticas y sus Aplicaciones\\
\noindent Universidad Sergio Arboleda, Bogot\'a, Colombia\\

\newpage

\begin{table}[!hbtp]
\centering
\resizebox{12cm}{0.5\textheight}{
\begin{tabular}{c}
\begin{tabular}{c|cccc}
\toprule
$\star$ & 0 & \{(1, 0)\} & \{(0, 1)\} & \{(0, 0)\} \\
\toprule
0 & 0 & 0 & 0 & 0 \\
\midrule
\{(1, 0)\} & 0 & \{(1, 0)\} & 0 & 0 \\
\midrule
\{(0, 1)\} & 0 & \{(0, 1)\} & 0 & 0 \\
\midrule
\{(0, 0)\} & 0 & 0 & \{(0, 1)\} & \{(0, 0)\} \\
\midrule
\{(1, 1)\} & 0 & 0 & \{(1, 0)\} & \{(1, 1)\} \\
\midrule
\{(1, 0), (0, 1)\} & 0 & \{(1, 0), (0, 1)\} & 0 & 0 \\
\midrule
\{(1, 0), (0, 0)\} & 0 & \{(1, 0)\} & \{(0, 1)\} & \{(0, 0)\} \\
\midrule
\{(1, 0), (1, 1)\} & 0 & \{(1, 0)\} & \{(1, 0)\} & \{(1, 1)\} \\
\midrule
\{(0, 1), (0, 0)\} & 0 & \{(0, 1)\} & \{(0, 1)\} & \{(0, 0)\} \\
\midrule
\{(0, 1), (1, 1)\} & 0 & \{(0, 1)\} & \{(1, 0)\} & \{(1, 1)\} \\
\midrule
\{(0, 0), (1, 1)\} & 0 & 0 & \{(1, 0), (0, 1)\} & \{(0, 0), (1, 1)\} \\
\midrule
\{(1, 0), (0, 1), (0, 0)\} & 0 & \{(1, 0), (0, 1)\} & \{(0, 1)\} & \{(0, 0)\} \\
\midrule
\{(1, 0), (0, 1), (1, 1)\} & 0 & \{(1, 0), (0, 1)\} & \{(1, 0)\} & \{(1, 1)\} \\
\midrule
\{(1, 0), (0, 0), (1, 1)\} & 0 & \{(1, 0)\} & \{(1, 0), (0, 1)\} & \{(0, 0), (1, 1)\} \\
\midrule
\{(0, 1), (0, 0), (1, 1)\} & 0 & \{(0, 1)\} & \{(1, 0), (0, 1)\} & \{(0, 0), (1, 1)\} \\
\midrule
\{(1, 0), (0, 1), (0, 0), (1, 1)\} & 0 & \{(1, 0), (0, 1)\} & \{(1, 0), (0, 1)\} & \{(0, 0), (1, 1)\} \\
\bottomrule
\end{tabular}\\
\\
\begin{tabular}{c|cccc}
\toprule
$\star$ & \{(1, 1)\} & \{(1, 0), (0, 1)\} & \{(1, 0), (0, 0)\} & \{(1, 0), (1, 1)\} \\
\toprule
0 & 0 & 0 & 0 & 0 \\
\midrule
\{(1, 0)\} & \{(1, 1)\} & \{(1, 0)\} & \{(1, 0)\} & \{(1, 0), (1, 1)\} \\
\midrule
\{(0, 1)\} & \{(0, 0)\} & \{(0, 1)\} & \{(0, 1)\} & \{(0, 1), (0, 0)\} \\
\midrule
\{(0, 0)\} & 0 & \{(0, 1)\} & \{(0, 0)\} & 0 \\
\midrule
\{(1, 1)\} & 0 & \{(1, 0)\} & \{(1, 1)\} & 0 \\
\midrule
\{(1, 0), (0, 1)\} & \{(0, 0), (1, 1)\} & \{(1, 0), (0, 1)\} & \{(1, 0), (0, 1)\} & \{(1, 0), (0, 1), (0, 0), (1, 1)\} \\
\midrule
\{(1, 0), (0, 0)\} & \{(1, 1)\} & \{(1, 0), (0, 1)\} & \{(1, 0), (0, 0)\} & \{(1, 0), (1, 1)\} \\
\midrule
\{(1, 0), (1, 1)\} & \{(1, 1)\} & 0 & \{(1, 0), (1, 1)\} & \{(1, 0), (1, 1)\} \\
\midrule
\{(0, 1), (0, 0)\} & \{(0, 0)\} & 0 & \{(0, 1), (0, 0)\} & \{(0, 1), (0, 0)\} \\
\midrule
\{(0, 1), (1, 1)\} & \{(0, 0)\} & \{(1, 0), (0, 1)\} & \{(0, 1), (1, 1)\} & \{(0, 1), (0, 0)\} \\
\midrule
\{(0, 0), (1, 1)\} & 0 & \{(1, 0), (0, 1)\} & \{(0, 0), (1, 1)\} & 0 \\
\midrule
\{(1, 0), (0, 1), (0, 0)\} & \{(0, 0), (1, 1)\} & \{(1, 0)\} & \{(1, 0), (0, 1), (0, 0)\} & \{(1, 0), (0, 1), (0, 0), (1, 1)\} \\
\midrule
\{(1, 0), (0, 1), (1, 1)\} & \{(0, 0), (1, 1)\} & \{(0, 1)\} & \{(1, 0), (0, 1), (1, 1)\} & \{(1, 0), (0, 1), (0, 0), (1, 1)\} \\
\midrule
\{(1, 0), (0, 0), (1, 1)\} & \{(1, 1)\} & \{(0, 1)\} & \{(1, 0), (0, 0), (1, 1)\} & \{(1, 0), (1, 1)\} \\
\midrule
\{(0, 1), (0, 0), (1, 1)\} & \{(0, 0)\} & \{(1, 0)\} & \{(0, 1), (0, 0), (1, 1)\} & \{(0, 1), (0, 0)\} \\
\midrule
\{(1, 0), (0, 1), (0, 0), (1, 1)\} & \{(0, 0), (1, 1)\} & 0 & \{(1, 0), (0, 1), (0, 0), (1, 1)\} & \{(1, 0), (0, 1), (0, 0), (1, 1)\} \\
\bottomrule
\end{tabular}\\
\\
\begin{tabular}{c|cccc}
\toprule
$\star$ & \{(0, 1), (0, 0)\} & \{(0, 1), (1, 1)\} & \{(0, 0), (1, 1)\} & \{(1, 0), (0, 1), (0, 0)\} \\
\toprule
0 & 0 & 0 & 0 & 0 \\
\midrule
\{(1, 0)\} & 0 & \{(1, 1)\} & \{(1, 1)\} & \{(1, 0)\} \\
\midrule
\{(0, 1)\} & 0 & \{(0, 0)\} & \{(0, 0)\} & \{(0, 1)\} \\
\midrule
\{(0, 0)\} & \{(0, 1), (0, 0)\} & \{(0, 1)\} & \{(0, 0)\} & \{(0, 1), (0, 0)\} \\
\midrule
\{(1, 1)\} & \{(1, 0), (1, 1)\} & \{(1, 0)\} & \{(1, 1)\} & \{(1, 0), (1, 1)\} \\
\midrule
\{(1, 0), (0, 1)\} & 0 & \{(0, 0), (1, 1)\} & \{(0, 0), (1, 1)\} & \{(1, 0), (0, 1)\} \\
\midrule
\{(1, 0), (0, 0)\} & \{(0, 1), (0, 0)\} & \{(0, 1), (1, 1)\} & \{(0, 0), (1, 1)\} & \{(1, 0), (0, 1), (0, 0)\} \\
\midrule
\{(1, 0), (1, 1)\} & \{(1, 0), (1, 1)\} & \{(1, 0), (1, 1)\} & 0 & \{(1, 1)\} \\
\midrule
\{(0, 1), (0, 0)\} & \{(0, 1), (0, 0)\} & \{(0, 1), (0, 0)\} & 0 & \{(0, 0)\} \\
\midrule
\{(0, 1), (1, 1)\} & \{(1, 0), (1, 1)\} & \{(1, 0), (0, 0)\} & \{(0, 0), (1, 1)\} & \{(1, 0), (0, 1), (1, 1)\} \\
\midrule
\{(0, 0), (1, 1)\} & \{(1, 0), (0, 1), (0, 0), (1, 1)\} & \{(1, 0), (0, 1)\} & \{(0, 0), (1, 1)\} & \{(1, 0), (0, 1), (0, 0), (1, 1)\} \\
\midrule
\{(1, 0), (0, 1), (0, 0)\} & \{(0, 1), (0, 0)\} & \{(0, 1), (0, 0), (1, 1)\} & \{(1, 1)\} & \{(1, 0), (0, 0)\} \\
\midrule
\{(1, 0), (0, 1), (1, 1)\} & \{(1, 0), (1, 1)\} & \{(1, 0), (0, 0), (1, 1)\} & \{(0, 0)\} & \{(0, 1), (1, 1)\} \\
\midrule
\{(1, 0), (0, 0), (1, 1)\} & \{(1, 0), (0, 1), (0, 0), (1, 1)\} & \{(1, 0), (0, 1), (1, 1)\} & \{(0, 0)\} & \{(0, 1), (0, 0), (1, 1)\} \\
\midrule
\{(0, 1), (0, 0), (1, 1)\} & \{(1, 0), (0, 1), (0, 0), (1, 1)\} & \{(1, 0), (0, 1), (0, 0)\} & \{(1, 1)\} & \{(1, 0), (0, 0), (1, 1)\} \\
\midrule
\{(1, 0), (0, 1), (0, 0), (1, 1)\} & \{(1, 0), (0, 1), (0, 0), (1, 1)\} & \{(1, 0), (0, 1), (0, 0), (1, 1)\} & 0 & \{(0, 0), (1, 1)\} \\
\bottomrule
\end{tabular}\\
\\
\begin{tabular}{c|cccc}
\toprule
$\star$ & \{(1, 0), (0, 1), (1, 1)\} & \{(1, 0), (0, 0), (1, 1)\} & \{(0, 1), (0, 0), (1, 1)\} & \{(1, 0), (0, 1), (0, 0), (1, 1)\} \\
\toprule
0 & 0 & 0 & 0 & 0 \\
\midrule
\{(1, 0)\} & \{(1, 0), (1, 1)\} & \{(1, 0), (1, 1)\} & \{(1, 1)\} & \{(1, 0), (1, 1)\} \\
\midrule
\{(0, 1)\} & \{(0, 1), (0, 0)\} & \{(0, 1), (0, 0)\} & \{(0, 0)\} & \{(0, 1), (0, 0)\} \\
\midrule
\{(0, 0)\} & \{(0, 1)\} & \{(0, 0)\} & \{(0, 1), (0, 0)\} & \{(0, 1), (0, 0)\} \\
\midrule
\{(1, 1)\} & \{(1, 0)\} & \{(1, 1)\} & \{(1, 0), (1, 1)\} & \{(1, 0), (1, 1)\} \\
\midrule
\{(1, 0), (0, 1)\} & \{(1, 0), (0, 1), (0, 0), (1, 1)\} & \{(1, 0), (0, 1), (0, 0), (1, 1)\} & \{(0, 0), (1, 1)\} & \{(1, 0), (0, 1), (0, 0), (1, 1)\} \\
\midrule
\{(1, 0), (0, 0)\} & \{(1, 0), (0, 1), (1, 1)\} & \{(1, 0), (0, 0), (1, 1)\} & \{(0, 1), (0, 0), (1, 1)\} & \{(1, 0), (0, 1), (0, 0), (1, 1)\} \\
\midrule
\{(1, 0), (1, 1)\} & \{(1, 1)\} & \{(1, 0)\} & \{(1, 0)\} & 0 \\
\midrule
\{(0, 1), (0, 0)\} & \{(0, 0)\} & \{(0, 1)\} & \{(0, 1)\} & 0 \\
\midrule
\{(0, 1), (1, 1)\} & \{(1, 0), (0, 1), (0, 0)\} & \{(0, 1), (0, 0), (1, 1)\} & \{(1, 0), (0, 0), (1, 1)\} & \{(1, 0), (0, 1), (0, 0), (1, 1)\} \\
\midrule
\{(0, 0), (1, 1)\} & \{(1, 0), (0, 1)\} & \{(0, 0), (1, 1)\} & \{(1, 0), (0, 1), (0, 0), (1, 1)\} & \{(1, 0), (0, 1), (0, 0), (1, 1)\} \\
\midrule
\{(1, 0), (0, 1), (0, 0)\} & \{(1, 0), (0, 0), (1, 1)\} & \{(1, 0), (0, 1), (1, 1)\} & \{(0, 1), (1, 1)\} & \{(1, 0), (1, 1)\} \\
\midrule
\{(1, 0), (0, 1), (1, 1)\} & \{(0, 1), (0, 0), (1, 1)\} & \{(1, 0), (0, 1), (0, 0)\} & \{(1, 0), (0, 0)\} & \{(0, 1), (0, 0)\} \\
\midrule
\{(1, 0), (0, 0), (1, 1)\} & \{(0, 1), (1, 1)\} & \{(1, 0), (0, 0)\} & \{(1, 0), (0, 1), (0, 0)\} & \{(0, 1), (0, 0)\} \\
\midrule
\{(0, 1), (0, 0), (1, 1)\} & \{(1, 0), (0, 0)\} & \{(0, 1), (1, 1)\} & \{(1, 0), (0, 1), (1, 1)\} & \{(1, 0), (1, 1)\} \\
\midrule
\{(1, 0), (0, 1), (0, 0), (1, 1)\} & \{(0, 0), (1, 1)\} & \{(1, 0), (0, 1)\} & \{(1, 0), (0, 1)\} & 0 \\
\bottomrule
\end{tabular}
\end{tabular}
}
\caption{Multiplication table for  $(\mathrm{DG}_{\mathrm{P}[1]}, \star)$.}
\label{table_product}
\end{table}

\end{document}